\documentclass{amsart}
\usepackage{amssymb}
\usepackage{amsthm}
\usepackage{indentfirst}
\usepackage{amsmath}
\usepackage{amscd}
\usepackage{xypic}

\numberwithin{equation}{section}

\newtheorem{Theorem}{Theorem}[section]
\newtheorem{Definition}[Theorem]{Definition}
\newtheorem{Proposition}[Theorem]{Proposition}
\newtheorem{Lemma}[Theorem]{Lemma}

\newtheorem{Assumption-Notation}[Theorem]{Assumption-Notation}

\newtheorem{Remark}[Theorem]{Remark}
\newtheorem{Corollary}[Theorem]{Corollary}

\newtheorem{Claim}[Theorem]{Claim}
\newtheorem{Fact}[Theorem]{Fact}
\newtheorem{Example}[Theorem]{Example}

\def\dim{\operatorname{dim}}

\begin{document}

\title[On the bicanonical map of primitive varieties with $q(X) = \dim X$]{On the bicanonical map of primitive varieties with $q(X) = \dim X$: the degree and the Euler number}
\address{Lei Zhang\\College of Mathematics and information Sciences\\Shaanxi Normal University\\Xi'an 710062\\P.R.China}
\email{lzhpkutju@gmail.com}
\author{Lei Zhang}
\begin{abstract}
Let $X$ be a variety of maximal Albanese dimension and of general type. Assume that $q(X) = \dim X$, the Albanese variety $\mathrm{Alb} (X)$ is a simple abelian variety, and the bicanonical map is not birational. We prove that the Euler number $\chi(X, \omega_X)$ is equal to 1, and $|2K_X|$ separates two distinct points over the same general point on $\mathrm{Alb} (X)$ via $\mathrm{alb}_X$ (Theorem \ref{main}).
\end{abstract}

\maketitle

\section{Introductions}\label{intro}

Let $X$ be a smooth complex projective variety of maximal Albanese dimension (m.A.d. for short) and of general type. Recall that tricanonical map is birational onto its image (cf. \cite{CH2} and \cite{JMT}). It is interesting to consider the birationality of
its bicanonical map. Let us recall the following results.

Assume moreover that the bicanonical map of $X$ is not birational. Then
\begin{itemize}
\item[I]{If $X$ is a surface, then either $X$ is fibered by curves of genus 2 (the standard case), or
\begin{enumerate}
\item[(i)]{if $q(X) >2$, then $X$ is birationally equivalent to a theta divisor of a principally polarized abelian variety (p.p.a.v. for short) of dimension 3 (cf. \cite{CCM});}
\item[(ii)]{if $q(X) = 2$, then $X$ is birational to a double cover of a simple principally polarized abelian surface $A$ branched
along a divisor $B \in |2\Theta|$ (cf. \cite{CCM}, \cite{CFM}, \cite{CM}).}
\end{enumerate}}
\item[II]{\begin{enumerate}
\item[(i)]{If $X$ is a primitive variety (cf. Def. \ref{prm}) with $q(X) > \dim X$, then it
is birational to a theta divisor of a p.p.a.v. (cf. \cite{BLNP}).}
\item[(ii)]{If $X$ is not necessarily primitive, then $\mathrm{gv}(\omega_X) \leq 1$, and the Albanese image is fibred by subvarieties
of codimension at most 1 of an Abelian subvariety of $\mathrm{Alb}(X)$ (cf. \cite{La}).}
\end{enumerate}}
\end{itemize}

If $X$ has a fibration $f: X \rightarrow Y$ with general fibers having non-birational bicanonical map, then the bicanonical map of $X$ is not birational. It is known that a non-primitive variety always has an irregular fibration by generic vanishing theorem (cf. Theorem \ref{gv}). Therefore, it is of special interest to study the bicanonical map of primitive varieties or those with simple Albanese varieties.

For the bicanonical map of primitive varieties, when $q(X) > \dim X$, it is completely clear by the results of I(i) and II(i); when $q(X) = \dim X$, it is not clear yet except in dimension 2 (I(ii)), and it is conjectured that if $\mathrm{Alb} (X)$ is simple, then $\mathrm{Alb}(X)$ is a p.p.a.v., and $X$ is birational to a double cover of $\mathrm{Alb}(X)$ branched
along a divisor $B \in |2\Theta|$ (see also \cite{La}).

In this paper, we study the case $q(X) = \dim X$, the main result is
\begin{Theorem}[Theorem \ref{eun}, \ref{spr2}]\label{main}
Let $X$ be a smooth complex projective variety of general type with $q(X) = \dim X$ and maximal Albanese dimension. Suppose that its bicanonical map is not birational and that $\mathrm{Alb}(X)$ is simple. Then $\chi(\omega_X) = 1$, and the linear system $|2K_X|$ separates two distinct points over the same general point $\mathrm{Alb} (X)$ via the Albanese map.
\end{Theorem}

This paper is organized as follows. In Section \ref{tool}, we list the technical results needed in this paper. In Section \ref{map}, we compare the Euler numbers of two irregular varieties of m.A.d. and equipped with a generically finite surjective morphism. In Section \ref{bicmap}, we study the bicanonical map and prove our main theorem. Finally, in Section \ref{inequ} as an appendix, we give an inequality on the irregularity of a fibration, and describe a certain fibration with the equality attained.

\textbf{Conventions:}

All varieties are assumed over $\mathbb{C}$.

``$\equiv$'' denotes the linear equivalence of line bundles or Cartier divisors respectively.

Let $E$ be a vector bundle on a variety $X$. We denote the projective bundle by $\mathbb{P}_X(E):=\mathrm{Proj}_{\mathcal{O}_X}(\oplus_kS^k(E^*))$ and the tautological line bundle by $\mathcal{O}_{\mathbb{P}_X(E)}(1)$.

Let $X$ be a projective variety. We denote by $D^b(X)$ be the
bounded derived category of coherent sheaves on $X$. Let $f: X \rightarrow Y$ be a morphism between two
smooth projective varieties. We denote by $Rf_*$ and $Lf^*$
the derived functors of $f_*$ and $f^*$ respectively.
We say an object $E \in D^b(X)$ is a sheaf if it is quasi-isomorphic to ($\cong$) a sheaf in $D^b(X)$.

For a product $X = X_1 \times X_2 \times ... \times X_r$ of $r$ varieties, $p_i$ denotes the projection from $X$ to the $i$\textsuperscript{th} factor $X_i$.

For an abelian variety $A$, $\hat{A}$ denotes its dual $\mathrm{Pic}^0 (A)$, $\mathcal{P}$ denotes the Poincar\'{e} line bundle on
$A \times \hat{A}$, and the Fourier-Mukai transform $R\Phi_{\mathcal{P}}: D^b(A) \rightarrow D^b(\hat{A})$ w.r.t. $\mathcal{P}$ is defined as
$$R\Phi_{\mathcal{P}}(\mathcal{F}) := R(p_2)_*(Lp_1^*\mathcal{F} \otimes \mathcal{P});$$
similarly $R\Psi_{\mathcal{P}}: D^b(\hat{A}) \rightarrow D^b(A)$ is defined as
$$R\Psi_{\mathcal{P}}(\mathcal{F}) := R(p_1)_*(Lp_2^*\mathcal{F} \otimes \mathcal{P}).$$
Since $p_i, i =1,2$ are flat morphisms, $R\Phi_{\mathcal{P}}$ and $R\Psi_{\mathcal{P}}$ are two right derived functors.

If $a: X \rightarrow A$ is a map to an abelian variety, then $\mathcal{P}_a : = (a \times \mathrm{id}_{\hat{A}}) ^*\mathcal{P}$, and for $\mathcal{F} \in D^b(X)$, $R\Phi_{\mathcal{P}_a}(\mathcal{F})$ is defined similarly; and if $\alpha \in \hat{A}$, we often denote the line bundle $a^*\alpha \in \mathrm{Pic}^0 (X)$ by $\alpha$ for simplicity. For an irregular variety $X$, we usually denote by $\mathrm{alb}_X: X \rightarrow
\mathrm{Alb} (X)$ the Albanese map.

{\bf Acknowledgements.} Part of this note appears in the author's
doctoral thesis submitted to Peking University (2011). The author expresses appreciations to Prof. Jinxing Cai and Dr. Wenfei Liu for many useful discussions. He thanks Prof. Meng Chen, Dr. Fan Peng and Ze Xu for their help on the inequality appearing in the appendix, and thanks Olivier Debarre and Yifei Chen for their suggestions in improving the English. He also thanks Sofia Tirabassi for some suggestions, and thanks the authors of \cite{BLNP} for their stimulating ideas. Finally, the author owes too much to an anonymous referee, who shares his or her ideas on improving the result of Theorem \ref{euln} and simplifying the proof of Corollary \ref{fm} and Theorem \ref{euln}. The author is supported by NSFC (No. 11226075).

\section{Definitions and technical results}\label{tool}

In this section, we collect some definitions and results needed in the sequel. First recall that

\begin{Theorem}[\cite{Mu} Thm. 2.2]\label{Mu} Let $A$ be an abelian variety of dimension $d$. Then
$$R\Psi_{\mathcal{P}} \circ R\Phi_{\mathcal{P}} = (-1)_A^*[-d]~\mathrm{and}~R\Phi_{\mathcal{P}} \circ R\Psi_{\mathcal{P}} = (-1)_{\hat{A}}^*[-d].$$
\end{Theorem}

\subsection{GV-sheaves, M-regular sheaves and $IT^0$-sheaves}
\begin{Definition}[\cite{PP2} Def. 2.1, 2.2, 2.8, 2.10, \cite{CH2} Def. 2.6]\label{defgv}
Given a coherent sheaf $\mathcal{F}$ on an abelian variety $A$, its \emph{$i$\textsuperscript{th} cohomological support locus} is defined as
$$V^i(\mathcal{F}): = \{\alpha \in \mathrm{Pic}^0 (A)| h^i(\mathcal{F} \otimes \alpha) > 0\}.$$
The number $\mathrm{gv}^i(\mathcal{F}): = \mathrm{codim}_{\mathrm{Pic}^0(A)}V^i(\mathcal{F}) - i$ is called the \emph{$i$\textsuperscript{th} generic vanishing index} of $\mathcal{F}$; $\mathrm{gv}(\mathcal{F}): = \min_{i>0}\{\mathrm{gv}^i(\mathcal{F})\}$ is called the \emph{generic vanishing index} of $\mathcal{F}$.
We say $\mathcal{F}$ is a \emph{GV-sheaf} (resp. \emph{M-regular sheaf}) if $\mathrm{gv}(\mathcal{F}) \geq 0$ (resp. $>0$) and an \emph{$IT^0$-sheaf} if $V^i(\mathcal{F}) = \emptyset$ for $i>0$.

Let $X$ be an irregular variety equipped with a morphism to an abelian variety $a: X \rightarrow
A$. Let $\mathcal{F}$ be a sheaf on $X$, its \emph{$i$\textsuperscript{th} cohomological support locus w.r.t. $a$} is defined as
$$V^i(\mathcal{F}, a) := \{\alpha \in \mathrm{Pic}^0(A)| h^i(X, \mathcal{F} \otimes \alpha) > 0\}$$
We say $\mathcal{F}$ is \emph{continuously~ globally~ generated} (\emph{CGG} for short) w.r.t. $a$ if the sum of the evaluation maps
$$\mathrm{ev}_U: \oplus_{\alpha \in U}H^0(\mathcal{F} \otimes \alpha) \otimes (\alpha^{-1}) \rightarrow \mathcal{F}$$
is surjective for any non-empty open set $U \subset \hat{A}$.
\end{Definition}

\begin{Proposition}[\cite{PP2} Thm. 5.1]\label{cgg}
An M-regular sheaf on an abelian variety is CGG.
\end{Proposition}

\begin{Proposition}\label{sjt}
Let $\mathcal{F}$ be a sheaf on an abelian variety $A$ of dimension $d$.
\begin{itemize}
\item[(i)]
If $\mathcal{F}$ is M-regular, then there is a natural surjection
$$(-1)_A^*R^d\Psi_\mathcal{P}R^0\Phi_\mathcal{P}\mathcal{F} \rightarrow \mathcal{F}.$$
\item[(ii)]
$\mathcal{F}$ is $IT^0$ if and only if $R\Phi_{\mathcal{P}}\mathcal{F} \cong R^0\Phi_{\mathcal{P}}\mathcal{F}$.
\end{itemize}
\end{Proposition}
\begin{proof}
(i) By assumption, for $j>0$, $\mathrm{codim}_{\hat{A}}\mathrm{Supp} R^j\Phi_\mathcal{P}\mathcal{F} > j$ (\cite{PP2} Prop. 2.1), thus
$$\clubsuit: R^i\Psi_\mathcal{P}R^j\Phi_\mathcal{P}\mathcal{F} = 0~\mathrm{if}~j \neq 0~ \mathrm{and} ~i+j \geq d.$$

By $(-1)_A^*R\Psi_\mathcal{P}R\Phi_\mathcal{P}(a_*\omega_X)\cong a_*\omega_X[-d]$ (Theorem~\ref{Mu}), applying Leray spectral sequence gives that
$$E_2^{i,j}:= (-1)_A^*R^i\Psi_\mathcal{P}R^j\Phi_\mathcal{P}\mathcal{F} \Rightarrow \mathcal{H}^{i+j}(\mathcal{F}[-d]).$$
Since $\clubsuit$, we have that
\begin{itemize}
\item
$E_{\infty}^{i,j} \cong 0$ for $i+j =d$ and $(i,j) \neq (d,0)$, thus
$E_{\infty}^{d,0} \cong \mathcal{F}$; and
\item
$d_r^{d,0}: E_{r}^{d,0} \rightarrow E_{r}^{d+r,-r+1} = 0$ is zero for $r\geq 2$, thus there is a surjection
$$E_{2}^{d,0} = (-1)_A^*R^d\Psi_\mathcal{P}R^0\Phi_\mathcal{P}\mathcal{F} \rightarrow E_{\infty}^{d,0}.$$
\end{itemize}

Then we conclude a natural surjection
$$(-1)_A^*R^d\Psi_\mathcal{P}R^0\Phi_\mathcal{P}\mathcal{F} \rightarrow \mathcal{F}.$$

(ii) The direction ``only if'' follows from applying \cite{ha77} Cor. 12.9. For the other direction, note that for every $\alpha \in \mathrm{Pic}^0 (A)$ and $i>d$ the natural map $R^i\Phi_\mathcal{P}\mathcal{F}\otimes \mathbb{C}(\alpha) \rightarrow H^i(A, \mathcal{F}\otimes \alpha) = 0$ is surjective, and if $R\Phi_{\mathcal{P}}\mathcal{F} \cong R^0\Phi_{\mathcal{P}}\mathcal{F}$, then $R^i\Phi_{\mathcal{P}}\mathcal{F} = 0$ for $i>0$. Then applying \cite{ha77} Cor. 12.11 (b), we can prove the direction ``if'' by induction.
\end{proof}

\begin{Corollary}\label{fm}
Let $a: X \rightarrow A$ be a generically finite morphism from a smooth projective variety to an abelian variety. Suppose that $\dim X = \dim A =d \geq 2$, $a^* : \mathrm{Pic}^0(A) \rightarrow \mathrm{Pic}^0(X)$ is an embedding, and for $i>0$, $V^i(\omega_X, a)$ is composed of at most some isolated points. Then there exists an exact sequence
$$(-1)_A^*R^d\Psi_\mathcal{P}R^0\Phi_\mathcal{P}(a_*\omega_X) \rightarrow a_*\omega_X \rightarrow \omega_A \rightarrow 0.$$
\end{Corollary}
\begin{proof}
By assumption we have a splitting (cf. for example \cite{CV} Prop. 1.2)
$$a_*\omega_X \cong \omega_A \oplus \mathcal{F}.$$
Then by $R\Phi_\mathcal{P}(a_*\omega_X) \cong R\Phi_\mathcal{P}\omega_A \oplus R\Phi_\mathcal{P} \mathcal{F}$ and $R^d\Phi_\mathcal{P}(a_*\omega_X) \cong R^d\Phi_\mathcal{P}\omega_A \cong \mathbb{C}(\hat{0})$ (Proposition 6.1 in \cite{BLNP}), we find that
$$R^i\Phi_\mathcal{P}(a_*\omega_X) \cong R^i\Phi_\mathcal{P} \mathcal{F}~\mathrm{for}~i = 0,1,...,d-1~\mathrm{and}~R^d\Phi_\mathcal{P} \mathcal{F}=0.$$
Therefore, $\mathcal{F}$ is M-regular, and there is a surjection by Proposition \ref{sjt} (i)
$$(-1)_A^*R^d\Psi_\mathcal{P}R^0\Phi_\mathcal{P}(a_*\omega_X) \cong (-1)_A^*R^d\Psi_\mathcal{P}R^0\Phi_\mathcal{P}\mathcal{F} \rightarrow \mathcal{F}.$$
Then naturally it follows the exact sequence
$$(-1)_A^*R^d\Psi_\mathcal{P}R^0\Phi_\mathcal{P}(a_*\omega_X) \rightarrow a_*\omega_X \rightarrow \omega_A \rightarrow 0$$
\end{proof}

Applying Theorem \ref{Mu} and Proposition \ref{sjt} (ii), we get (see \cite{Zh} Cor. 2.2 for details)
\begin{Corollary}\label{Muc}
Let $A$ be an abelian variety of dimension $d$, $E$ an $IT^0$-vector bundle on $A$. Then $R\Phi_{\mathcal{P}}E$ is a vector bundle on $\hat{A}$, and its dual $(R\Phi_{\mathcal{P}}E)^*$ is an $IT^0$-vector bundle such that
$$R\Psi_{\mathcal{P}}((R\Phi_{\mathcal{P}}E)^*)  \cong E^* \cong ((-1)_A^*R^d\Psi_{\mathcal{P}}((R^0\Phi_{\mathcal{P}}E)))^*.$$
\end{Corollary}

\subsection{Generic vanishing theorem}
Recall generic vanishing theorem due to Green and Lazarsfeld:
\begin{Theorem} [\cite{GL1}, \cite{GL2}]\label{gv} Let $X$ be a smooth compact k\"{a}hler manifold with $\dim X =n$ and $\dim \mathrm{alb}_X(X) = k$. Then
\begin{enumerate}
\item[(i)]{$\mathrm{codim}_{\mathrm{Pic}^0 (X)}V^i(\omega_X, \mathrm{alb}_X) \geq k-n +i$.}
\item[(ii)]{Let $Z$ be a component of $V^i(\omega_X, \mathrm{alb}_X)$ of positive dimension. Then $Z$ is a subtorus of $\mathrm{Pic}^0 (X)$, and there exists an analytic variety $Y$ of dimension $\leq n-i$ and a dominant map $f: X \rightarrow Y$ such that $Z \subset \alpha + f^*\mathrm{Pic}^0 (Y)$ where $\alpha$ is torsion.}
\item[(iii)]{$Y$ has maximal Albanese dimension.}
\item[(iv)]{Let $\alpha \in V^i(\omega_X, \mathrm{alb}_X)$ and $v \in H^1(\mathcal{O}_X) = T_\alpha \mathrm{Pic}^0(X)$. The \emph{derived complex} below

 $$\centerline{\xymatrix{
&H^{n-i-1}(\alpha^{-1}) \ar[r]^{\cup v} &H^{n-i}(\alpha^{-1}) \ar[r]^{\cup v} &H^{n-i+1}(\alpha^{-1})
}}$$
is exact if $v$ is not contained in the tangent cone $TC_\alpha V^i(\omega_X, \mathrm{alb}_X)$ to $V^i(\omega_X, \mathrm{alb}_X)$ at $\alpha$.}
\end{enumerate}
\end{Theorem}

\begin{Corollary}\label{rmk}
Let $X$ be a smooth projective variety of m.A.d. and of dimension $d$. Then
\begin{itemize}
\item[(i)]{$h^i(X, \omega_X \otimes \alpha) = h^i(\mathrm{Alb}(X), (\mathrm{alb}_X)_*\omega_X \otimes \alpha)$ for $\alpha \in \mathrm{Pic}^0 (X), i \geq 0$, and $(\mathrm{alb}_X)_*\omega_X $ is a GV-sheaf, thus
    $$V^0(\omega_X, \mathrm{alb}_X) \supset V^1(\omega_X, \mathrm{alb}_X) \supset \cdot\cdot\cdot \supset V^d(\omega_X, \mathrm{alb}_X);$$}
\item[(ii)]{$R\Phi_{\mathcal{P}}((\mathrm{alb}_X)_*\mathcal{O}_X)[d] \cong R^d\Phi_{\mathcal{P}}((\mathrm{alb}_X)_*\mathcal{O}_X)$ is a sheaf, which we denote by $\widehat{\mathcal{O}_X}$;}
\item[(iii)]{$R^i\Phi_{\mathcal{P}}((\mathrm{alb}_X)_*\omega_X) \cong (-1)^*_{\mathrm{Pic}^0(X)}\mathcal{E}xt^i(\widehat{\mathcal{O}_X}, \mathcal{O}_{\mathrm{Pic}^0(X)})$;}
\item[(iv)]{$p_g(X) > \chi(\omega_X)$.}
\end{itemize}
\end{Corollary}
\begin{proof}
By Koll\'{a}r's results (\cite{Ko1} Thm. 2.1, \cite{Ko2} Thm. 3.1), we have that
$$R(\mathrm{alb}_X)_*\omega_X \cong \sum_iR^i(\mathrm{alb}_X)_*\omega_X[-i]$$
and $R^i(\mathrm{alb}_X)_*\omega_X$ is torsion free if restricted to the Albanese image $\mathrm{alb}_X(X)$. We conclude that $R^i(\mathrm{alb}_X)_*\omega_X = 0$ for $i>0$ since $\mathrm{alb}_X$ is generically finite, hence $R(\mathrm{alb}_X)_*\omega_X \cong (\mathrm{alb}_X)_*\omega_X$. Using Grothendieck duality and projection formula, the assertions (i), (ii) and (iii) follow from Theorem \ref{gv}. See \cite{Ha} Thm. 1.5, 4.1 and Cor. 3.2 for the details.

For (iv), take a general $v \in H^1(\mathcal{O}_X) = T_{\hat{0}} \mathrm{Pic}^0X$. Theorem \ref{gv} (iv) tells that \emph{the derived complex} $D_v$ is exact
$$\centerline{\xymatrix{&0 \ar[r] &H^{0}(\mathcal{O}_X) \ar[r]^{\cup v} &H^{1}(\mathcal{O}_X) \ar[r]^{\cup v}&\cdot\cdot\cdot\ar[r]^{\cup v} &H^{n-1}(\mathcal{O}_X) \ar[r]^{\cup v}&H^{n}(\mathcal{O}_X)
}}$$
This implies that the cokernel of the right-most map is a linear space of dimension $\chi(\omega_X) = (-1)^n\chi(\mathcal{O}_X)$. Since $X$ is of m.A.d., the right-most map is non-zero. Therefore, $p_g(X) = h^n(\mathcal{O}_X) > \chi(\omega_X)$.
\end{proof}

\begin{Remark}[\cite{EL} Remark 1.6]
If replacing the Albanese map by a generically finite morphism to an abelian variety $a: X \rightarrow A$ and replacing $\mathrm{Pic}^0(X)$ by $\mathrm{Pic}^0(A)$, then the evident analogues of the results in Corollary \ref{rmk} hold.
\end{Remark}

\begin{Proposition}\label{cnm}
Let $a: X \rightarrow A$ be a generically finite morphism from a smooth projective variety onto an abelian variety $A$. Suppose that $\chi(\omega_X) > 0$. Then for any $n>0$ the pluri-canonical map $\phi_{nK_X}$ does not factor through  $a$ rationally.
\end{Proposition}
\begin{proof}
We only need to consider the canonical map. Since $X$ is of m.A.d., we have
$p_g(X) > 1$ by Corollary \ref{rmk} (iv), and thus the canonical map $\phi_{K_X}$ is not constant.

Assume to the contrary that $\phi_{K_X} = g \circ a$ where $g: A \dashrightarrow \mathbb{P}^{p_g(X) - 1}$. By blowing up $X$ and $A$, we get a birational model of $g \circ a: X \rightarrow A \dashrightarrow \mathbb{P}^{p_g(X) - 1}$
$$\tilde{g} \circ \tilde{a}: \tilde{X} \rightarrow \tilde{A} \rightarrow \mathbb{P}^{p_g(X) - 1}$$
such that both $\tilde{g}$ and $\tilde{a}$ are morphisms and
$$|K_{\tilde{X}}| = (\tilde{g} \circ \tilde{a})^* |\mathcal{O}_{\mathbb{P}^{p_g(X) - 1}}(1)| + F = \tilde{a}^*|M| + F~ \mathrm{where}~ |M| = \tilde{g}^* |\mathcal{O}_{\mathbb{P}^{p_g(X) - 1}}(1)|$$

Denote by $\tilde{R}$ the ramification divisor of $\tilde{g}: \tilde{X} \rightarrow \tilde{A}$. Then $K_{\tilde{X}} \equiv \tilde{R} + \tilde{a}^*E$ where $E$ is an effective divisor on $\tilde{A}$ exceptional w.r.t. the blowing up map $\tilde{A} \rightarrow A$. So there exists $M \in  |M|$ such that $\tilde{R} + \tilde{a}^*E - \tilde{a}^*M$ is an effective divisor. Notice that $M$ is not contained in $E$. We get a contradiction by the property of the ramification divisor.
\end{proof}

\begin{Definition}[\cite{Ca} Def. 1.24]\label{prm}
Let $X$ be an irregular variety of m.A.d.. It is called \emph{primitive} if $V^i(\omega_X, \mathrm{alb}_X)$ is composed of at most finitely many points for $i>0$.
\end{Definition}

\subsection{Characterization of a theta divisor}
Imitating the proof of \cite{BLNP} Prop. 3.1, we can prove
\begin{Proposition}\label{refp}
Let $X$ be a smooth projective variety of general type equipped with a generically finite morphism $a: X \rightarrow A$ to an abelian variety $A$. Suppose that
\begin{itemize}
\item[(i)]{$\dim V^1(\omega_X, a) = 0$;}
\item[(ii)]{$\dim X < \dim A$ and $a^*: \mathrm{Pic}^0(A) \rightarrow \mathrm{Pic}^0(X)$ is an embedding; and}
\item[(iii)]{$\chi(X, \omega_X) = 1$.}
\end{itemize}
Then A is a p.p.a.v., and $a: X \rightarrow A$ birationally maps $X$ to a theta divisor on $A$.
\end{Proposition}

\begin{Corollary}\label{ref}
Let $X$ and $a: X \rightarrow A$ be as in Proposition \ref{refp}. Assume (i), (ii) in Proposition \ref{refp} and
\begin{itemize}
\item[(iii)']{for $\alpha \in U_0:=\hat{A} \setminus V^1(\omega_X, a)$, $|\omega_X \otimes \alpha| = |M| + F_\alpha$ where $M$ is the movable part which is independent of $\alpha$ and $F_\alpha$ is the fixed part.}
\end{itemize}
Then A is a p.p.a.v., and $a: X \rightarrow A$ birationally maps $X$  to a theta divisor on $A$.
\end{Corollary}
\begin{proof}
Assumption (iii)' implies that $\mathcal{B}:= \{(x, \alpha) \in X \times U_0| x \in F_\alpha\}$ is a divisor in $X \times U_0$. Denote by $\bar{\mathcal{Y}}$ the closure of $\mathcal{B}$ in $X \times \hat{A}$. Noticing that $\mathrm{codim}_{\hat{A}}V^1(\omega_X, a)\geq 2$, by the see-saw principle we have (\cite{BLNP} Lemma 5.2)
$$\bar{\mathcal{Y}} \equiv p_1^* \omega_X(-M) \otimes \mathcal{P}_a \otimes p_2^*\mathcal{O}_{\hat{A}}(\bar{\mathcal{Y}}_p)$$
where $p \in X$ is a point mapped to $0 \in A$ via $a$.

With these settings, by similar argument as in \cite{BLNP} Lemma 5.3, we can show that $\chi(X, \omega_X) = 1$. Then we are done by Proposition \ref{refp}.
\end{proof}

\subsection{Universal divisors and separation} \label{rzh}
Recall the following results from \cite{Zh} Sec. 3.
\begin{Theorem}[\cite{Zh} Theorem 2.10]\label{pf}
Let $X$ and $Y$ be two projective normal varieties, and $\mathcal{L}$ a line bundle on $X \times Y$. Assume $E= (p_2)_*\mathcal{L}$ is a vector bundle and put $P = \mathbb{P}_Y(E)$.
Note that there exists an open set $U \subset P$ parametrizing the divisors in $|\mathcal{L}_y|, y \in Y$. Denote by $\mathcal{D} \subset X \times U$ the universal family. Then its closure $\bar{\mathcal{D}} \subset X
\times P$ is a divisor, and
$$\bar{\mathcal{D}} \equiv p^*\mathcal{L} \otimes q^*\mathcal{O}_P(1)$$
where $p,q$ denote the two projections $p: X \times P \rightarrow X
\times Y$, $q: X \times P \rightarrow P$.
\end{Theorem}

Let $E$ be an $IT^0$-vector bundle on an abelian variety $A$. Then $R\Phi_{\mathcal{P}}E$ is a vector bundle $\hat{A}$. Its dual $(R\Phi_{\mathcal{P}}E)^*$ is an $IT^0$-vector bundle, and $R\Psi_{\mathcal{P}}(R\Phi_{\mathcal{P}}(E)^*) \cong E^*$ (cf. Corollary \ref{Muc}).

Let $P = \mathbb{P}_A(E^*)$, $\hat{P} =
\mathbb{P}_{\hat{A}}(R\Phi_{\mathcal{P}}(E))$, and denote by $\pi: P \rightarrow A$ and $\hat{\pi}: \hat{P} \rightarrow \hat{A}$ the natural projections. Note that
$$(p_2)_*(p_1^*\mathcal{O}_{P}(1)\otimes (\pi \times \mathrm{id}_{\hat{A}})^*\mathcal{P}) \cong  R\Phi_{\mathcal{P}}(E)~\mathrm{and}~(p_1)_*(p_2^*\mathcal{O}_{\hat{P}}(1)\otimes (\mathrm{id}_{A} \times \hat{\pi})^*\mathcal{P}) \cong  E^*.$$
We can identify $\hat{P}$ (resp. $P$) with the Hilbert scheme parametrizing the divisors in $\{|\mathcal{O}_{P}(1)\otimes \alpha||\alpha \in \mathrm{Pic}^0(P) = \hat{A}\}$ (resp. $\{|\mathcal{O}_{\hat{P}}(1)\otimes \hat{\alpha}||\hat{\alpha} \in \mathrm{Pic}^0(\hat{P}) = A\}$). Denote by
$\mathcal{U} \subset P \times \hat{P}$ the universal family and  by $\tilde{\mathcal{P}}$
the pull-back $(\pi \times \hat{\pi})^*\mathcal{P}$ of the Poincar\'{e} bundle on $A \times \hat{A}$. We have
\begin{itemize}
\item
$\mathcal{U} \equiv p_1^* \mathcal{O}_P(1) \otimes \tilde{\mathcal{P}} \otimes p_2^*\mathcal{O}_{\hat{P}}(1) $ (by Theorem \ref{pf});
\item
identifying a divisor in $|\mathcal{O}_{P}(1)\otimes \alpha|,\alpha \in \hat{A}$ with a point in $\hat{P}$, for every $x \in P$, the fiber $\mathcal{U}_x$ parametrizes all those divisors passing through $x$;
\item
for $x,y \in P$, $\mathcal{U}_x \equiv \mathcal{U}_y \Leftrightarrow \pi(x) = \pi(y), ~\mathrm{and}~\mathcal{U}_x = \mathcal{U}_y \Leftrightarrow x = y$.
\end{itemize}

We can write that
\begin{equation}\label{dec}
\mathcal{U}_x = \mathcal{H}_x + \mathcal{V}_x ~\text{and}~
\mathcal{V}_x= \mathcal{V}^1_x + \cdots + \mathcal{V}^r_x
\end{equation}
where $\mathcal{H}_x$ is
the horizontal part (if $\mathrm{rank}(R\Phi_{\mathcal{P}}(E)) = 1$ then $\mathcal{H}_x = \emptyset$), $\mathcal{V}_x = \hat{\pi}^*V_x$ is the vertical part ($\mathcal{V}_x = \emptyset$ if $\mathcal{U}_x$ is irreducible), and the
$\mathcal{V}^i_x = \hat{\pi}^*V^i_x$'s are the reduced and
irreducible vertical components (two of them may equal).
In fact there is a decomposition $\mathcal{U}= \mathcal{H} + \mathcal{V}$ such that for general $x \in P$, $\mathcal{U}_x = \mathcal{H}_x + \mathcal{V}_x$.

\begin{Lemma}[\cite{Zh}, Lemma 3.3]\label{spr}
Let $x,y \in P$ be two distinct points. Write that $\mathcal{U}_x = \mathcal{H}_x +
\mathcal{V}_x $ and $\mathcal{U}_y = \mathcal{H}_y +
\mathcal{V}_y$ as in \ref{dec}. Then the following conditions are equivalent
\begin{itemize}
\item[(a)]{$|\mathcal{O}_P(2)|$ fails to separate $x,y$;}
\item[(b)]{$\mathcal{H}_x= \mathcal{H}_y$ and $\mathrm{Supp}(V_x +(-1)_{\hat{A}}^*V_x) = \mathrm{Supp}(V_y +(-1)_{\hat{A}}^*V_y)$.}
\end{itemize}
\end{Lemma}

\section{The maps between two irregular varieties}\label{map}

Here we give a theorem comparing the Euler numbers of two varieties of m.A.d. and equipped with a generically finite surjective morphism.
Similar result has been proved by Tirabassi with a stronger assumption (\cite{Ti} Prop. 5.2.4). A weaker version also appeared in \cite{CLZ}, where it is applied to study the automorphism groups inducing trivial actions on cohomology of irregular varieties.

\begin{Theorem}\label{euln}
Let $\pi: X \rightarrow Z$ be a generically finite surjective morphism between two smooth projective varieties of m.A.d.. Then $\chi(\omega_X) \geq \chi(\omega_Z)$.

If moreover
\begin{itemize}
\item[(i)]{$\pi$ is not birational;}
\item[(ii)]{$\pi^*: \mathrm{Pic}^0 (Z) \rightarrow \mathrm{Pic}^0 (X)$ is an embedding; and}
\item[(iii)]{$\mathrm{gv}^i(\omega_X, \mathrm{a}_X)\geq 1~\mathrm{for} ~i=1,2,\cdots,\dim X-1$, where $\mathrm{a}_X:= \mathrm{alb}_Z \circ \pi: X \rightarrow Z \rightarrow \mathrm{Alb} (Z)$,}
\end{itemize}
then $\chi(X, \omega_X) > \chi(Z, \omega_Z)$.
\end{Theorem}
\begin{proof}
By assumption we have a splitting $\pi_*\omega_X \cong \omega_Z \oplus \mathcal{F}$. Since $(\mathrm{a}_X)_*\omega_X$ is a GV-sheaf on $\mathrm{Alb} (Z)$, the direct summand $(\mathrm{alb}_Z)_*\mathcal{F}$ is also a GV-sheaf. Then
$$\chi(Z, \mathcal{F}) = \chi(\mathrm{Alb} (Z), (\mathrm{alb}_Z)_*\mathcal{F}) = h^0(\mathrm{Alb} (Z), (\mathrm{alb}_Z)_*\mathcal{F} \otimes \alpha) \geq 0 ~\mathrm{for~ general}~ \alpha \in \mathrm{Pic}^0 (Z),$$
and it follows that
$$\chi(X, \omega_X) =  \chi(Z, \omega_Z) + \chi(Z, \mathcal{F}) \geq \chi(Z, \omega_Z).$$

Now assume (i, ii, iii). Note that

(i) implies that $\mathcal{F} \neq 0$;

(ii) implies that $R^d \Phi_\mathcal{P}((\mathrm{a}_X)_*\omega_X) \cong R^d \Phi_\mathcal{P}((\mathrm{alb}_Z)_*\omega_Z) \cong \mathbb{C}(\hat{0})$ where $d = \dim X$, thus $R^d \Phi_\mathcal{P}((\mathrm{alb}_Z)_*\mathcal{F}) = 0$;

(iii) implies that $\mathrm{gv}^i((\mathrm{alb}_Z)_*\mathcal{F}) \geq 1~\mathrm{for} ~i=1,2,\cdots,\dim X-1$.

So we conclude that $(\mathrm{alb}_Z)_*\mathcal{F}$ is a non-zero M-regular sheaf. Since $(\mathrm{alb}_Z)_*\mathcal{F}$ is CGG (cf. Proposition \ref{cgg}), for general $\alpha \in \mathrm{Pic}^0 (Z)$, we have
$$\chi(\mathrm{Alb} (Z), (\mathrm{alb}_Z)_*\mathcal{F}) = h^0(\mathrm{Alb} (Z), (\mathrm{alb}_Z)_*\mathcal{F} \otimes \alpha) > 0.$$
As a consequence we get that
$$\chi(X, \omega_X) > \chi(Z, \omega_Z).$$
\end{proof}

\section{The bicanonical map}\label{bicmap}

\begin{Assumption-Notation}\label{not2}
Let $X$ be a smooth projective variety of general type and of m.A.d., with $q(X) = \dim X = d\geq 2$. Denote by $a: X \rightarrow A$ the Albanese
map, and assume $A$ is simple, which implies that $X$ is primitive. Suppose that the bicanonical map $\phi: X \dashrightarrow \mathbb{P}^{P_2(X) - 1}$ is not birational.
\end{Assumption-Notation}

\subsection{The Fourier-Mukai transform of $\omega_X$}
\begin{Lemma}\label{pre}
$R^0\Phi_{\mathcal{P}_a}(\omega_X)  \cong
\mathcal{O}_{\hat{A}}(-\hat{D})^{\oplus \chi(\omega_X)}$ where
$\hat{D}$ is an ample divisor on $\hat{A}$.
\end{Lemma}
\begin{proof}
Let $U_0 = \hat{A} \setminus V^1(\omega_X, a)$ and $\mathcal{B}_a(x) = \{\alpha \in U_0| x~ \mathrm{is ~a ~ base~ point~ of }~ |\omega_X \otimes \alpha|\}$. Applying \cite{BLNP} Theorem 4.13 gives that $\mathrm{codim}_{\hat{A}}\mathcal{B}_a(x) = 1$ for general $x \in X$. Denote by $\bar{\mathcal{Y}}$ the divisorial part of the closure of $\mathcal{B}:= \{(x, \alpha) \in X \times U_0| \alpha \in B_a(x)\}$ in $X\times \hat{A}$. We conclude that for $\alpha \in U_0$, $|\omega_X \otimes \alpha| = |M_\alpha| + F_\alpha$, where $|M_\alpha|$ is the movable part and $F_\alpha = \bar{\mathcal{Y}}_\alpha$ is the fixed part. As in  \cite{BLNP} Sec. 5.1, we define a map $f: \hat{A} \rightarrow \hat{A}$. Since $\hat{A}$ is simple, we conclude that $f = \mathrm{id}_{\hat{A}}$ by \cite{BLNP} Lemma 5.1 (a). As a consequence $|M_\alpha|$ is independent of $\alpha$., i.e.,
$$|\omega_X \otimes \alpha| = |M| + F_\alpha$$
By \cite{BLNP} Lemma 5.2, we have
$$\mathcal{P}_a \cong \mathcal{O}_{X \times \hat{A}}(\bar{\mathcal{Y}}) \otimes p_1^*(\omega_X^{-1} \otimes M) \otimes p_2^*\mathcal{O}_{\hat{A}}(-\hat{D})$$
where $\hat{D}$ is a fiber $\bar{\mathcal{Y}}_p$ for some $p \in X$.
Then there is an exact sequence
$$0 \rightarrow \mathcal{P}_a^{-1} \rightarrow p_1^*(\omega_X \otimes M^{-1})\otimes p_2^*\mathcal{O}_{\hat{A}}(\hat{D}) \rightarrow p_1^*(\omega_X \otimes M^{-1})\otimes p_2^*\mathcal{O}_{\hat{A}}(\hat{D})|_{\bar{\mathcal{Y}}} \rightarrow 0$$
Applying $R^d(p_2)_*$ to the sequence above, we obtain the following exact sequence
\begin{equation}\label{3}
0 \rightarrow \tau \rightarrow (-1)_{\hat{A}}^*\widehat{\mathcal{O}_X} \rightarrow \mathcal{O}_{\hat{A}}(\hat{D})^{\oplus \chi(\omega_X)} \rightarrow \tau' \rightarrow 0
\end{equation}
where
\begin{enumerate}
\item[(a)]{The rank $\chi(\omega_X)$ in the third term appears because $h^d(\omega_X \otimes M^{-1}) = h^0(M) = \chi(\omega_X)$.}
\item[(b)]{$\tau'$ is supported at the locus of the $\alpha \in \hat{A}$ such that the fiber $\bar{\mathcal{Y}}_\alpha$ of the projection $p_2: \bar{\mathcal{Y}} \rightarrow \hat{A}$ has dimension $d$. Such locus is contained in $V^1(\omega_X,a)$, hence consists of finitely many torsion points.}
\item[(c)]{Since $(-1)_{\hat{A}}^*\widehat{\mathcal{O}_X}$ (cf. Corollary \ref{rmk} (ii)) also has rank $\chi(\omega_X)$, the kernel of the map $(-1)_{\hat{A}}^*\widehat{\mathcal{O}_X} \rightarrow \mathcal{O}_{\hat{A}}(\hat{D})^{\oplus \chi(\omega_X)}$ is the torsion part $\tau \cong \mathbb{C}(\hat{0})$ of $(-1)_{\hat{A}}^*\widehat{\mathcal{O}_X}$ (cf. \cite{BLNP} Prop. 6.1). }
\item[(d)]{Note that
$R\Psi_{\mathcal{P}}(\tau'), R\Psi_{\mathcal{P}}(\tau)~\mathrm{and}~
R\Psi_{\mathcal{P}}((-1)_{\hat{A}}^*\widehat{\mathcal{O}_X}) \cong a_*\mathcal{O}_X$ (by Theorem \ref{Mu}) are all sheaves on
$A$. Applying $R\Psi_{\mathcal{P}}$ to Sq. \ref{3}, then by using spectral sequence we conclude that $R\Psi_{\mathcal{P}}(\mathcal{O}_{\hat{A}}(\hat{D})^{\oplus
\chi(\omega_X)})$ is also a sheaf, hence $\hat{D}$ is an ample
divisor on $\hat{A}$.}
\end{enumerate}
Note that
$$\mathcal{E}xt^i(\mathcal{O}_{\hat{A}}(\hat{D})^{\oplus \chi(\omega_X)}, \mathcal{O}_{\hat{A}}) = 0 ~\mathrm{if}~i\neq 0~\mathrm{and}~ \mathcal{E}xt^i(\tau, \mathcal{O}_{\hat{A}}) = \mathcal{E}xt^i(\tau', \mathcal{O}_{\hat{A}}) = 0~\mathrm{if}~i\neq d.$$
Recall that  $R^i\Phi_{\mathcal{P}_a}(\omega_X) \cong (-1)_{\hat{A}}^*\mathcal{E}xt^i(\widehat{\mathcal{O}_X}, \mathcal{O}_{\hat{A}})$ (cf. Corollary \ref{rmk} (iii)). Applying $\mathcal{E}xt(-, \mathcal{O}_{\hat{A}})$ to (\ref{3}), by using spectral sequence, we conclude that
$$R^0\Phi_{\mathcal{P}_a}(\omega_X) \cong (-1)_{\hat{A}}^*\mathcal{E}xt^0(\widehat{\mathcal{O}_X}, \mathcal{O}_{\hat{A}}) \cong \mathcal{O}_{\hat{A}}(-\hat{D})^{\oplus \chi(\omega_X)}$$
\end{proof}

\subsection{The universal divisor}\label{ud}
Since $\hat{D}$ is an ample divisor on $\hat{A}$, $\mathcal{O}_{\hat{A}}(\hat{D})^{\oplus \chi(\omega_X)}$ is an $IT^0$-vector bundle, using Corollary \ref{Muc}, we have that the sheaf
$$E := (-1)_A^*R^d\Psi_{\mathcal{P}}R^{0}\Phi_{\mathcal{P}_a}(\omega_X) \cong (-1)_A^*R^d\Psi_{\mathcal{P}}(\mathcal{O}_{\hat{A}}(-\hat{D})^{\oplus \chi(\omega_X)}) \cong (R\Psi_{\mathcal{P}}(\mathcal{O}_{\hat{A}}(\hat{D})^{\oplus \chi(\omega_X)}))^*$$
is an $IT^0$-vector bundle such that $R^0\Phi_{\mathcal{P}}(E) = R^{0}\Phi_{\mathcal{P}_a}(\omega_X)$. By Corollary \ref{fm}, $E$ fits into the following exact sequence
\begin{equation}\label{4}
E \rightarrow a_*\omega_X \rightarrow \omega_A \rightarrow 0.
\end{equation}

Let
$$\hat{P} =
\mathbb{P}_{\hat{A}}(R^{0}\Phi_{\mathcal{P}_a}(\omega_X))=\mathbb{P}_{\hat{A}}(R^0\Phi_{\mathcal{P}}(E)) \cong \hat{A} \times \mathbb{P}^{\dim |M|}~\mathrm{and}~P = \mathbb{P}_A(E^*).$$
Then $\hat{P}$ is one component of the Hilbert scheme
parametrizing the divisors in $|K_X \otimes \alpha|, \alpha \in \hat{A}$. We denote by $\mathcal{K}
\subset X \times \hat{P}$
the universal family (cf. Theorem \ref{pf}). Let the notation $\tilde{\mathcal{P}}$, $\mathcal{U}$, $\pi$ and $\hat{\pi}$ be as in Sec. \ref{rzh}.
By the proof of Lemma \ref{pre}, for $\hat{p} = (\alpha, M) \in U_0 \times \mathbb{P}^{\dim |M|} \subset \hat{P}$, where $M\in |M|$, we have
$$\mathcal{K}_{\hat{p}} = \bar{\mathcal{Y}}_\alpha + M.$$
So $(\mathrm{id}_X \times \hat{\pi})^*\bar{\mathcal{Y}} \subset \mathcal{K}$, and we can write that
\begin{equation}
\mathcal{K} = \mathcal{H} + (\mathrm{id}_X \times \hat{\pi})^*\bar{\mathcal{Y}}.
\end{equation}

\begin{Fact}\label{facts}
\begin{itemize}
\item[(a)]
If $\mathcal{H}$ is non-empty (i.e., $\chi(X, \omega_X) > 1$), then it is dominant over $\hat{P}$, and for $\hat{p} = (\alpha, M) \in \hat{P}$, $\mathcal{H}_{\hat{p}} = M$;
\item[(b)]
for general $x \in X$, $\mathcal{H}_x = \hat{A} \times H_x \subset \hat{A} \times \mathbb{P}^{\dim |M|}$ where $H_x$ is the hyperplane in $\mathbb{P}^{\dim |M|}$ parametrizing all the divisors in $|M|$ passing through $x$;
\item[(c)]
for the divisor $\bar{\mathcal{Y}}$, denoting by $\mathcal{V}$ the sum of all the components dominant over $X$, then $\bar{\mathcal{Y}} = \mathcal{V} + p_1^*F$ where $F \subset X$ is the common fixed part of all $F_\alpha, \alpha \in U_0$;
\item[(d)]
for a general point $x \in X \setminus F$, $\mathcal{Y}_x = \mathcal{V}_x\equiv \mathcal{O}_{\hat{A}}(\hat{D}) \otimes \mathcal{P}_{a(x)}$.
\end{itemize}
\end{Fact}

Let $\tilde{\mathcal{P}}_a = (a \times \hat{\pi})^*\mathcal{P}$. By Theorem \ref{pf}, we have
$$\mathcal{K} \equiv p_1^*\omega_X \otimes \tilde{\mathcal{P}}_a \otimes p_2^*\mathcal{O}_{\hat{P}}(1)~~~\mathrm{and}~~~  \mathcal{U} \equiv p_1^* \mathcal{O}_P(1) \otimes \tilde{\mathcal{P}}\otimes p_2^*\mathcal{O}_{\hat{P}}(1).$$
Observe that for a general $x \in X$, the fiber $\mathcal{K}_x$ is a divisor on $\hat{P}$ linearly equivalent to $\mathcal{O}_{\hat{P}}(1) \otimes \hat{\pi}^*\mathcal{P}_{a(x)}$, hence is a fiber of $\mathcal{U} \rightarrow P$. We can define a rational map relative over $A$
$$h: X \dashrightarrow P ~\mathrm{via}~x \mapsto \mathcal{K}_x,$$
Assume that $h$ is a morphism by blowing up $X$. There exists an open set $U \subset X$ such that the restriction $\mathcal{K}|_{U \times \hat{P}} = (h \times \mathrm{id}_{\hat{P}})^* \mathcal{U}|_{U \times \hat{P}}$. Since $\mathcal{U}$ is flat over $P$, $(h \times \mathrm{id}_{\hat{P}})^* \mathcal{U}$ is the closure of $(h \times \mathrm{id}_{\hat{P}})^* \mathcal{U}|_{U \times \hat{P}}$ in $X \times \hat{P}$, thus $(h \times \mathrm{id}_{\hat{P}})^* \mathcal{U} \subset \mathcal{K}$.
\begin{Fact}\label{fact}
\begin{itemize}
\item[(1)] Using the see-saw principle, $\mathcal{K} = (h \times id_{\hat{P}})^*\mathcal{U} \otimes p_1^*\mathcal{O}_X(G)$ where $G$ is an effective divisor on $X$ such that
$h^*\mathcal{O}_P(1) + G \equiv \omega_X$.
\item[(2)] We have a natural homomorphism $\otimes s: h^*\mathcal{O}_P(1) \rightarrow \omega_X$ where $s \in H^0(X, \mathcal{O}_X(G))$ is a section with zero locus $G$, then pushing forward gives a homomorphism $a_*h^*\mathcal{O}_P(1) \rightarrow a_*\omega_X$.
\item[(3)]
The relative map $h: X/A \rightarrow P/A$ is induced by the homomorphism $E =\pi_*\mathcal{O}_P(1)  \rightarrow
a_*h^*\mathcal{O}_P(1)$.
\item[(4)] The composite homomorphism $E \rightarrow a_*h^*\mathcal{O}_P(1) \rightarrow a_*\omega_X$ coincides with the natural homomorphism $E \rightarrow a_*\omega_X$ in (\ref{4}) up to multiplication by a non-zero constant.
\end{itemize}
\end{Fact}
We explain (4). Since $E$ is CGG, the composite homomorphism $E \rightarrow a_*h^*\mathcal{O}_P(1) \rightarrow a_*\omega_X$ is determined by its Fourier-Mukai transform $\lambda: R^0\Phi_{\mathcal{P}}E \rightarrow R^0\Phi_{\mathcal{P}}(a_*h^*\mathcal{O}_P(1)) \rightarrow R^0\Phi_{\mathcal{P}}(a_*\omega_X)$. By abuse of notation, we also use $\mathcal{U}$ and $\mathcal{K}$ for the line bundles on $P \times \hat{P}$ and $X \times \hat{P}$ linearly equivalent to $\mathcal{U}$ and $\mathcal{K}$ respectively. Then with the corresponding terms being isomorphic, we have that $\lambda$ coincides with the following natural composite homomorphism
$$\lambda': \hat{\pi}_*(p_2)_*(\mathcal{U} \otimes p_2^*\mathcal{O}_{\hat{P}}(-1)) \rightarrow \hat{\pi}_*(p_2)_*((h \times \mathrm{id}_{\hat{P}})^*\mathcal{U} \otimes p_2^*\mathcal{O}_{\hat{P}}(-1)) \rightarrow \hat{\pi}_*(p_2)_*(\mathcal{K} \otimes p_2^*\mathcal{O}_{\hat{P}}(-1)).$$
We can see that $\lambda'$ coincides with the Fourier-Mukai transform of $E \rightarrow a_*\omega_X$ in (\ref{4}) up to multiplication by a non-zero constant, then (4) follows.

\begin{Lemma}\label{embd}
If $\deg(a) > 2$, then $h : X \rightarrow P$ is an embedding generically, which means that for two general distinct points $x, y \in X$, $\mathcal{K}_x \neq \mathcal{K}_y$.
\end{Lemma}
\begin{proof}
By Fact \ref{fact} (3, 4), the degree of the restriction map $\pi|_{h(X)}: h(X) \rightarrow A$ is $\geq \mathrm{rank}(E \rightarrow a_*h^*\mathcal{O}_P(1)) \geq \deg(a)-1$. Then we have $\deg(h) \leq
\frac{\deg(a)}{\deg(a) -1}$, and the assertion
follows easily by assumption.
\end{proof}

\subsection{The Euler number $\chi(X, \omega_X)$}

\begin{Proposition}\label{ir}
$\mathcal{V}$ is irreducible.
\end{Proposition}
\begin{proof}
We argue by contradiction. Suppose that $\mathcal{V} = \mathcal{V}^1 + \mathcal{V}^2$ is reducible. Note that both $\mathcal{V}^1$ and $\mathcal{V}^2$ are dominant over $X$ and $\hat{A}$. Fixing a general $\alpha_0$, we define two maps $\iota_i: \hat{A} \rightarrow \hat{A}$ via $\alpha \mapsto \mathcal{V}^i_\alpha - \mathcal{V}^i_{\alpha_0}$ with $\mathcal{V}^i_\alpha - \mathcal{V}^i_{\alpha_0}$ identified as an element in $\hat{A} = \mathrm{Pic}^0(X)$, which extend to two morphisms. We claim that neither of the two maps are constant. Indeed, say, if $\iota_1$ is constant, then for a general $\alpha \in \hat{A}$ $\mathcal{V}^1_\alpha \equiv \mathcal{V}^1_{\alpha_0}$ and $\mathcal{V}^1_\alpha \neq \mathcal{V}^1_{\alpha_0}$, so $\mathcal{V}^1_\alpha + \mathcal{V}^2_{\alpha_0} \equiv \mathcal{V}^1_{\alpha_0} + \mathcal{V}^2_{\alpha_0} \in |\mathcal{V}_{\alpha_0}|$, contradicting that $\mathcal{V}_{\alpha_0}$ is contained in the fixed part of $|\omega_X \otimes \alpha_0|$. Therefore, both $\iota_1$ and $\iota_2$ are surjective since $\hat{A}$ is simple. For general $\alpha \in \hat{A}$, there exist $\alpha_1, \alpha_2 \in \hat{A}$ such that $\iota_1(\alpha_1) = \alpha, \iota_2(\alpha_2) = \alpha^{-1}$. Then we conclude that
$$\mathcal{V}_{\alpha_0} = \mathcal{V}^1_{\alpha_0} + \mathcal{V}^2_{\alpha_0} \equiv \mathcal{V}^1_{\alpha_1} + \mathcal{V}^2_{\alpha_2}$$
which contradicts that $\mathcal{V}_{\alpha_0}$ is contained in the fixed part of $|\omega_X \otimes \alpha_0|$ again.
\end{proof}

\begin{Theorem}\label{eun}
The Euler number $\chi(\omega_X) = 1$.
\end{Theorem}
\begin{proof}
We argue by contradiction. Suppose that $\chi(\omega_X) \geq 2$. So for general $\alpha \in \mathrm{Pic}^0 (X)$, the movable part $|M|$ of $|\omega_X \otimes \alpha|$ is non-trivial. By taking two different general elements $M_1, M_2 \in |M|$, we define a rational map $f: X \dashrightarrow \mathbb{P}^1$, and assume that $f$ is a morphism by blowing up $X$. Let $f = \pi \circ g: X \rightarrow Y \rightarrow \mathbb{P}^{1}$ be the Stein factorization.

Since $A$ is simple and $\dim A \geq 2$, we conclude that $Y$ is a rational curve. Take a general fiber $M'$ of $g: X \rightarrow Y$. Then $M'$ is smooth, $M \equiv kM'$ for some integer $k>0$, and the restriction map of the bicanonical map $\phi|_{M'}$ is not birational.

We claim that $M'$ is not birational to a theta divisor on a p.p.a.v.. Indeed, otherwise we have $q(M') = \dim M' +1 = \dim X$, thus  $q(X) = q(M') + q(Y)$. Then $X$ is birational to $M' \times \mathbb{P}^1$ by Theorem \ref{cp}, contradicting that $X$ is of general type.

It is reduced to prove that
\begin{Claim}
$M'$ is birational to a theta divisor on a p.p.a.v..
\end{Claim}
\emph{Proof of the claim:} We break the proof into 3 steps.

\underline{Step 1}: Consider the line bundle $\omega_X(M')$. For general $x\in X$, we define the locus $B'_x := \{\alpha \in \hat{A}|x~ \mathrm{is ~a ~base ~ point~of}~|\omega_X(M')\otimes \alpha|\}$. Then $\mathrm{codim}_{\hat{A}}B'_x = 1$.

Assume to the contrary that $\mathrm{codim}_{\hat{A}}B'_x > 1$. Then take two general distinct points $x,y \in X$ such that $|2K_X|$ fails to separate them. We can see that every $M \in |M|$ containing $x$ must contain $y$, thus $\mathcal{H}_x = \mathcal{H}_y$ by Fact \ref{facts} (b).

We claim that $\mathcal{V}_x = \mathcal{V}_y$, as a consequence $\mathcal{K}_x = \mathcal{K}_y$ and $a(x) = a(y)$ by Fact \ref{facts} (d).
Indeed, if not, we can choose $\alpha \in \hat{A}$ contained in $\mathcal{V}_x$ while not in $\mathcal{V}_y$ such that $-\alpha$ is not contained in $B'_y$ since $\mathrm{codim}_{\hat{A}}(B'_y) > 1$. It follows that $\mathcal{V}_\alpha$ contains $x$ but not $y$, and there exists $D \in |\omega_X(M')\otimes \alpha^{-1}|$ which does not contain $y$. Taking a divisor $M' \in |M'|$ not containing $y$, then the divisor $\mathcal{V}_\alpha + D + (k-1)M' + F \in |2K_X|$ contains $x$ but not $y$ (where $F$ is introduced in Fact \ref{facts} (c)), a contradiction.

Then we obtain a contradiction by Proposition \ref{cnm} if $\deg (a) = 2$, and by Lemma \ref{embd} if $\deg (a) > 2$.

\underline{Step 2}: $|\omega_X(M')\otimes \alpha| = |H| + F'_{\alpha}$, where the movable part $H$ is independent of general $\alpha\in \hat{A}$.

Since $\mathrm{codim}_{\hat{A}}B'_x = 1$ for general $x \in X$, similarly as in the proof of Lemma \ref{pre} we get a divisor $\bar{\mathcal{Y}}'$ dominant over $\hat{A}$, such that $|\omega_X(M')\otimes \alpha| = |H_{\alpha}| + F'_{\alpha}$ for general $\alpha \in \hat{A}$, where $|H_{\alpha}|$ is the movable part and $F'_{\alpha} = \bar{\mathcal{Y}}'_\alpha$ is the fixed part. Since $|M'|$ is base point free, we have $F'_{\alpha} \leq F_\alpha$, i.e., $\bar{\mathcal{Y}}'_\alpha \leq \bar{\mathcal{Y}}_\alpha$. There exists a non-empty open set $U \subset \hat{A}$ such that, the restriction $\bar{\mathcal{Y}}'|_{X \times U} \leq \bar{\mathcal{Y}}|_{X \times U}$, thus $\bar{\mathcal{Y}}' \leq \bar{\mathcal{Y}}$ because they are the closure of $\bar{\mathcal{Y}}'|_{X \times U}$ and $\bar{\mathcal{Y}}|_{X \times U}$ in $X \times \hat{A}$ respectively. Denote by $\mathcal{V}'$ the sum of the components of $\bar{\mathcal{Y}}'$ dominant over $X$. We have $\mathcal{V}' \leq \mathcal{V}$, thus $\mathcal{V}' = \mathcal{V}$ by Proposition \ref{ir}. Denote by $F'$ the common fixed part of $|\omega_X(M')\otimes \alpha|$. We can see that $H_{\alpha} \equiv \omega_X(M' - \mathcal{V}_\alpha -F')\otimes \alpha$ is independent of general $\alpha\in \hat{A}$. Setting $F'_{\alpha} = \mathcal{V}_\alpha + F'$, then we are done.

\underline{Step 3}: Tensoring the following exact sequence
$$0\rightarrow \omega_X \rightarrow \omega_X(M') \rightarrow \omega_{M'} \rightarrow 0$$
with $\alpha \in U_0$ and taking cohomology, we conclude that the restriction map $H^0(X, \omega_X(M')\otimes \alpha) \rightarrow H^0(M', \omega_{M'}\otimes \alpha)$ is surjective since $H^1(X,\omega_X \otimes \alpha) = 0$. Then by $|\omega_X(M')\otimes \alpha| = |H| + F'_{\alpha}$, we have that
$$|\omega_{M'}\otimes \alpha| = |H||_{M'} +  F'_{\alpha}|_{M'}$$
By assumption that $A$ and $\hat{A}$ are simple, they have no proper subtorus of positive dimension, we conclude that $A$ is generated by the translates through the origin of $a(M')$, and $\dim V^1(\omega_{M'}, a|_{M'}) = 0$ by generic vanishing theorem.

The restriction morphism $a|_{M'}: M'\rightarrow A$ factors through a morphism to an abelian variety $a_{M'}: M'\rightarrow A_{M'}$, where $A_{M'}$ is the dual of the image of the natural map $(a|_{M'})^*: \mathrm{Pic}^0(A) \rightarrow \mathrm{Pic}^0 (M')$. So $(a_{M'})^*:\mathrm{Pic}^0(A_{M'}) \rightarrow \mathrm{Pic}^0(M')$ is an embedding. Write that $ a|_{M'} =\eta \circ a_{M'}: M'\rightarrow A_{M'} \rightarrow A$. Then $\eta$ is finite, and thus $\eta^*: \mathrm{Pic}^0(A) \rightarrow  \mathrm{Pic}^0(A_{M'})$ is an epimorphism. So $V^1(\omega_{M'}, a_{M'}) = \eta^*V^1(\omega_{M'}, a|_{M'})$ is of dimension $0$. Applying Corollary \ref{ref} to $a_{M'}: M'\rightarrow A_{M'}$ shows that $M'$ is birational to a theta divisor.
\end{proof}

\begin{Remark}
To prove $\chi(X, \omega_X) = 1$, the simplicity of $\mathrm{Alb}(X)$ is necessary by Example \ref{eg}.
\end{Remark}

\subsection{The degree of the bicanonical map}\label{sdeg}

$\chi(X, \omega_X) = 1$ implies that $R^0\Phi_{\mathcal{P}_a}(\omega_X)  \cong
\mathcal{O}_{\hat{A}}(-\hat{D})$ is a line bundle, and $\hat{P} = A$. By Fact \ref{facts}, we have that $\mathcal{K} = \bar{\mathcal{Y}}$, and for $x \in X\setminus F$, $\mathcal{K}_x = \bar{\mathcal{Y}}_x= \mathcal{V}_x$. Write that
$\mathcal{V}_x = V^1_x + ...+V^r_x$ as in Sec. \ref{rzh}.
\begin{Theorem}\label{deg}
Let the notation be as above. Then $\deg(\phi) \leq 2^r$.
\end{Theorem}
\begin{proof}
By Lemma \ref{spr}, if $x, y \in P$ are two distinct points such that $|\mathcal{O}_P(2)|$ fails to separate them, then
$\mathcal{H}_x = \mathcal{H}_y$
and
$$\mathcal{V}_y =  ((-1)_{\hat{A}}^{\epsilon_1})^*V^1_x + \cdots +
((-1)_{\hat{A}}^{\epsilon_r})^*V^r_x, ~for~some~\epsilon_i \in \{0,1\}, i = 1,2,...,r$$
which has $2^r$ possibilities.

If $\deg(h) = 1$, then we are done since $h^*|\mathcal{O}_P(2)| \subset |2K_X|$.

If $\deg(h) > 1$, then $\deg(a) = 2$ by Lemma \ref{embd}, and thus $a$ and $h$ are birationally equivalent. The assertion follows by combining the two facts that the restriction of $|\mathcal{O}_P(2)|$ on $h(X)$ defines a map of degree $\leq 2^r$ and that the bicanonical map does not factor through $a$ rationally (cf. Proposition \ref{cnm}).
\end{proof}

\begin{Theorem}\label{spr2}
$|2K_X|$ separates the points over the same general point on $A$.
\end{Theorem}
\begin{proof}
Consider the diagonal map $(a \times \phi): X \dashrightarrow A \times \mathbb{P}^{P_2(X) - 1}$. We can assume this map is a morphism by blowing up $X$, and denote by $Z$ the image. If the theorem is not true, then $X \rightarrow Z$ is not birational. Note that $a$ factors through $(a \times \phi): X \rightarrow Z$, so $(a \times \phi)^*: \mathrm{Pic}^0(Z) \rightarrow \mathrm{Pic}^0(X)$ is an embedding. Since $\chi(\omega_X) = 1$, Theorem \ref{euln} implies that $\chi(\omega_Z) = 0$, so $Z$ is birational to $A$ (\cite{BLNP}, Prop. 4.10). Therefore, $(a \times \phi): X \rightarrow Z$ is birational to $a: X \rightarrow A$, and $\phi$ factors through $a$ rationally. However, this contradicts Proposition \ref{cnm}.
\end{proof}

\subsection{Remarks and an example}\label{sre}

We remark the following:

(1) $\mathcal{V}$ is irreducible (Proposition \ref{ir}), it is expected that $\mathcal{V}_x$ is irreducible for general $x \in X$. If this is true, then by Theorem \ref{deg}, the bicanonical map $\phi$ is of degree 2. Precisely, using the idea of \cite{Zh}, we know that $\phi$ factors through an involution $\sigma$, and up to a translate on $A$, the quotient map $X \rightarrow X/(\sigma)$ fits into the following commutative diagram
\[\begin{CD}
X      @> >>      X/(\sigma) \\
@Va VV               @Va' VV \\
A       @> >>    A/((-1)_A)
\end{CD} \]

(2) For a primitive variety $X$, if we do not assume $A= \mathrm{Alb}(X)$ is simple, then $R^0\Phi_{\mathcal{P}_a} \omega_X \cong (\mathcal{E}(\mathcal{D}))^*$ where $\mathcal{E}$ is a vector bundle and $\mathcal{D}$ is a divisor on $\mathrm{Pic}^0 (X)$ (cf. \cite{BLNP} proof of Lemma 5.3). Assume that $\mathcal{E}(\mathcal{D})$ is an $IT^0$-vector bundle. Then $E = (R^0\Psi_{\mathcal{P}}\mathcal{E}(\mathcal{D}))^*$ is an $IT^0$-vector bundle on $A$. Similarly we can define $P, \hat{P}, \mathcal{K}$ and $\mathcal{V}$. For general $x \in X$, if $\mathcal{V}_x = \hat{\pi}^*(V^1_x + \cdots+V^r_x)$ as before, then by similar argument we can prove $\deg(\phi) \leq 2^r$. This bound is analogous to \cite{Zh} Corollary 3.4, and is optimal (cf. Example \ref{eg}). Stimulated by \cite{Zh} Theorem 1.2 and 3.5, it is expected that $A$ is decomposable if $r>1$ (Example \ref{eg} provides an evidence). So it is possible to give an upper bound to $\deg(\phi)$ relying on the numbers of the factors of $A$.

\begin{Example}\label{eg}
Let
\begin{itemize}
\item
$(A_i, \Theta_i) , i = 1,2,\cdots,r$ be $r$ simple p.p.a.v. and $A_{r+1}$ a simple abelian variety;
\item
$A = A_1 \times A_2 \times \cdots \times A_r \times A_{r+1}$;
\item
$B = p_1^*B_1 + \cdots + p_r^*B_r + p_{r+1}^*B_{r+1}$ where $B_i \in |2\Theta_i|$ is a smooth divisor on $A_i$ for $i =1,2,\cdots,r$ and $B_{r+1}\equiv 2D_{r+1}$ is a smooth very ample divisor on $A_{r+1}$;
\item
$Y \rightarrow A$ the double cover given by the relation $2L \equiv B$ where $L$ is a line bundle linearly equivalent to $p_1^*\Theta_1 + \cdots + p_r^*\Theta_r + p_{r+1}^*D_{r+1}$;
\item
$X \rightarrow Y$ a smooth resolution.
\end{itemize}
Note that $Y$ has at most canonical singularities since $B$ is a reduced simple normal crossing divisor. We denote by $\pi: X \rightarrow Y \rightarrow A$ the composed map which coincides with the Albanese map.

Immediately we have
\begin{itemize}
\item[(i)]
$\omega_X \equiv \pi^*L$, thus $\pi_*\omega_X \cong \omega_A \oplus L$ and $X$ is primitive;
\item[(ii)]
$\pi_*\omega_X^2 \cong L \oplus \mathcal{O}_A(B)$, and the linear system $|B|$ defines a morphism of degree $2^r$ on $A$;
\item[(iii)]
$E \cong L$ is a line bundle, $P = A$ and $R\Phi_{\mathcal{P}} E$ is a vector bundle of rank $\chi(X, \omega_X) = \chi(A_{r+1}, \mathcal{O}_{A_{r+1}}(D_{r+1}))$.
\end{itemize}

We can prove that (with details left to readers)
\begin{itemize}
\item[(1)]
for general $x \in X$, $\mathcal{K}_x = \mathcal{H}_x + \mathcal{V}_x^1 + \mathcal{V}_x^2 + \cdots + \mathcal{V}_x^r$ where $\mathcal{V}_x^i$ is the pull-back of the divisor on $\hat{A}_i$ parametrizing the theta divisors passing through $p_i (\pi (x))$ via the projection $\hat{P} \rightarrow \hat{A} \rightarrow \hat{A}_i$ for $i = 1,2,\cdots,r$;
\item[(2)]
the degree of the bicanonical map of $X$ is $2^r$.
\end{itemize}
\end{Example}

\section{Appendix: An inequality on the fibrations of irregular varieties}\label{inequ}
Let $f:S\rightarrow C$ be a fibration of a smooth surface and $F$ a general fiber. Beauville proved that $q(S) \leq q(F) + q(C)$; and if $q(F) \geq 2$, then the equality is attained if, and only if $S$ is birational to $C \times F$ (\cite{Beau}).

For arbitrary dimensional case, we have
\begin{Theorem}\label{ineq}
Let $f: X \rightarrow Y$ be a fibration between two smooth projective varieties and $F$ a general fiber. Then $q(X) \leq q(Y) + q(F)$, and the kernel of the restriction map $r: \mathrm{Pic}^0(X) \rightarrow \mathrm{Pic}^0(F)$ contains $f^*\mathrm{Pic}^0(Y)$, as the whole component passing through the identity point $\hat{0} \in \mathrm{Pic}^0(X)$.
\end{Theorem}
\begin{proof}
The inequality has been proved in \cite{CH1} Cor. 3.5 for Iitaka fibration. We use the notation in \cite{CH1} Lemma 2.6 and Cor. 3.5 for convenience. Noticing that the natural map $\mathrm{A}(X_y) \rightarrow J$ is surjective by the proof of \cite{CH1} Lemma 2.6, the inequality $q(X) \leq q(Y) + q(F)$ is obtained by applying the proof of \cite{CH1} Cor. 3.5 straightforwardly. The remaining assertion follows from \cite{CH1} Lemma 2.6 iii).

Another approach is using Beauville's argument (\cite{Beau}) and \cite{Lan} Chap. VIII Theorem 13.
\end{proof}

\begin{Theorem}\label{cp}
Let $f: X \rightarrow Y$ be a fibration between two smooth projective varieties. Suppose that for general $y \in Y$, $\mathrm{Alb}(X_y)$ is a p.p.a.v., the general fiber $X_y$ is birational to a theta divisor $F_y \subset \mathrm{Alb}(X_y)$, and that $q(X) = q(X_y) + q(Y)$. Then $\mathrm{Alb}(X_y)$ is isomorphic to a p.p.a.v. $A$ independent of general $y \in Y$, $F_y\cong F$ where $F$ is a theta divisor on $A$, and $X$ is birational to $F \times Y$.
\end{Theorem}
\begin{proof}
By Theorem \ref{ineq} and the assumption $q(X) = q(X_y) + q(Y)$, the restriction map $\mathrm{Pic}^0(X)/f^*\mathrm{Pic}^0(Y) \rightarrow \mathrm{Pic}^0(X_y)$ is an epimorphism between two abelian varieties of the same dimension, and the kernel consists of finitely many torsion points which is independent of general $y$. Then we can see that $\mathrm{Pic}^0(X_y)$ is independent of general $y$ up to isomorphisms, so is its dual $\mathrm{Alb}(X_y)$. We can assume $\mathrm{Alb}(X_y) \cong A$, and $A$ has a theta divisor $F$ such that $F_y\cong F$.

Note that $f: X \rightarrow Y$ has a birational model $f': X' \rightarrow Y$ such that the general fibers are all isomorphic to $F$. Take an equivariant resolution $\tilde{f}: \tilde{X} \rightarrow Y$ of $f': X' \rightarrow Y$ (\cite{Ka} p.14), whose general fibers are smooth and isomorphic to each other. Let $\tilde{F}$ be a general fiber of $\tilde{f}$. Since $\tilde{F}$ is of general type, using \cite{Le} Proposition 1, we know that $\tilde{f}: \tilde{X} \rightarrow Y$ is birational to $(\tilde{F} \times Z)/G \rightarrow Z/G$, where $G$ is a finite group and the action of $G$ on the product $\tilde{F} \times Z$ is compatible with the actions on each factor. The action $G$ on $\tilde{F}$ descends to $F$, and since $q(X) = q(F) + q(Y)$, $G$ induces a trivial action on $H^1(F, \mathcal{O}_F)$.

If we can show $G$ acts on $F$ trivially, then we are done. From now on fix the Albanese map $a: F \rightarrow A$, and take $\sigma \in G$. By the universal property of Albanese maps, we obtain the following commutative diagram

$$\centerline{\xymatrix{
&F   \ar[d]^{a} \ar[r]^\sigma     &F  \ar[d]^{a}\\
&A   \ar[r]^{\tilde{\sigma}}   &A
}}$$
i.e., $\sigma$ extends to an isomorphism $\tilde{\sigma}$ of $A$ fixing $F$.

Since $\sigma$ acts trivially on $H^1(F, \mathcal{O}_{F})$, $\tilde{\sigma}$ acts trivially on $H^1(A, \mathcal{O}_{A})$, so it is nothing but a translate $t_a$ for some $a \in A$. Since $F$ is a theta divisor, the morphism $\phi_{F}: A \rightarrow \mathrm{Pic}^0(A)$ via $a' \mapsto t_{a'}^*F - F$ is an isomorphism. Then since $\tilde{\sigma}$ fixes $F$, we have $t_a^*F = F$, thus $a = 0$, this means $\sigma$ is identity.
\end{proof}


\begin{thebibliography}{1dffs}
\bibitem[BLNP]{BLNP} M.A. Barja, M. Lahoz, J.C. Naranjo and G. Pareschi, On the bicanonical map of irregular varieties, J. Algebraic Geom. 21 (2012), 445--471.
\bibitem[Beau]{Beau} Beauville, A.: L'inegalites $p_g \geq 2q -4$
pour les surfaces de type general. Bull. Soc. Math. France 110,
343--346 (1982). Appendix to: Inegalites numeriques pour les
surfaces de type general.
\bibitem[CLZ]{CLZ} J. Cai, W. Liu, L. Zhang, Automorphisms of surfaces of general type with $q \geq 2$ acting trivially in cohomology, Compositio. Math. (2013), (doi: 10.1112/S0010437X13007264).
\bibitem[CV]{CV} J. Cai, E. Viehweg, Irregular manifolds whose canonical system is composed of a pencil, Asian J. Math. 8 (2004), 027--038.
\bibitem[Ca]{Ca} F. Catanese, Moduli and classification of irregular Kahler manifolds (and algebraic varieties) with Albanese general type, Invent. Math. 104 (1991), 263--289.
\bibitem[CCM]{CCM} F. Catanese, C. Ciliberto, M. Mendes Lopses, On the classification of irregular surfaces of general type with nonbirational bicanonical map, Trans. Amer. Math. Soc. 350 (1998), no. 1, 275--308.
\bibitem[CFM]{CFM} C. Ciliberto, P. Francia, M. Mendes Lopes, Remarks on the bicanonical map for surfaces
of general type, Math. Z. 224 (1997), 137--166.
\bibitem[CH1]{CH1} J.A. Chen, C.D. Hacon, On the irregularity of the image of the Iitaka fibration. Commun. in Alg. 32 (2004), 203--215.
\bibitem[CH2]{CH2} J.A. Chen, C.D. Hacon, Pluricanonical systems on irregular 3-folds of general type, Math.
Z. 255 (2007), no. 2, 343--355.
\bibitem[CM]{CM} C. Ciliberto, M. Mendes Lopes, On surfaces with $p_g = q = 2$ and non-birational bicanonical
map, Adv. Geom. 2 (2002), 281--300.
\bibitem[EL]{EL} L. Ein, R. Lazarsfeld, Singularities of theta divisors, and birational geometry of irregular
varieties, J. Amer. Math. Soc. 10 (1997), 243--258.
\bibitem[GL1]{GL1} M. Green, R. Lazarsfeld, Deformation theory, Generic vanishing theorems and some conjectures of Enriques, Catanese and Beauville, Invent. math. 90 (1987), 389--407.
\bibitem[GL2]{GL2} M. Green, R. Lazarsfeld, Higher obstruction to deformation of cohomology of line bundles, J. Amer.
Math. Soc. 4 (1991), 87--103.
\bibitem[Ha]{Ha} C. D. Hacon, A derived approach to generic vanishing, J. Reine Angew. Math., 575 (2004),
173--187.
\bibitem[Har]{ha77} R. Hartshorne, Algebraic Geometry, Graduate Texts in Mathematics, No. 52. 1977.
\bibitem[JMT]{JMT} Z. Jiang, M. Lahoz, S. Tirabassi, On the Iitaka fibration of varieties of maximal Albanese dimension, Int. Math. Res. Notices (2012) (doi: 10.1093/imrn/rns131).
\bibitem[Ka]{Ka} Y. Kawamata, Minimal models and Kodaira dimension of algbraic fiber spaces,  J. Reine Angew. Math., 363 (1985), 1--46.
\bibitem[Ko1]{Ko1} J. Koll\'{a}r, Higher direct images of dualizing sheaves I, Ann. Math. 123 (1986), 11--42.
\bibitem[Ko2]{Ko2} J. Koll\'{a}r, Higher direct images of dualizing sheaves II, Ann. Math. 124 (1986), 171--202.
\bibitem[Lah]{La} M. Lahoz, Generic vanishing index and the birationality of the bicanonical map of irregular varieties, Math. Z., 272 (2012), 1075--1086.
\bibitem[Lan]{Lan} S. Lang, Abelian varieties, Interscience Wiley, New York, 1959.
\bibitem[Laz]{Laz} R. Lazarsfeld, Positivity in algebraic geometry  II, Ergebnisse der Mathematik und ihrer Grenzgebiete, 49, Springer-Verlag, Berlin, 2004.
\bibitem[Mu]{Mu} S. Mukai, Duality between $D(X)$ and $D(\hat{X})$ with its application to Picard sheaves, Nagoya Math. J. 81 (1981), 153--175.
\bibitem[PP]{PP2}G. Pareschi and M. Popa, Regularity on abelian varieties III: relationship with Generic Vanishing and applications, Clay. Math. Proc., 2006, 141--168.
\bibitem[Le]{Le} M. Levine, Deformation of irregular threefolds, Lect. Notes in Math. 947,
Springer (1982) 269--286.
\bibitem[Ti]{Ti} S. Tirabassi, Ph.D. thesis, Universit $\grave{a}$ degli studi Roma TRE (http://arxiv.org/abs/1210.0324).
\bibitem[Zh]{Zh} L. Zhang, A note on the linear systems on the projective bundles over Abelian varieties, to appear in Proc. Amer.
Math. Soc. (http://arxiv.org/abs/1209.2939).

\end{thebibliography}
\end{document}